\documentclass[10pt, A4paper,reqno]{amsart}
\usepackage{amsmath,amsfonts,amssymb}
\usepackage[breaklinks]{hyperref}
\usepackage{graphicx}
\usepackage{longtable}
\makeatletter
\@namedef{subjclassname@2020}{%
	\textup{2020} Mathematics Subject Classification}
\makeatother

\theoremstyle{plain}
\newtheorem{thm}{Theorem}[section]
\newtheorem{lem}{Lemma}[section]
\newtheorem{cor}{Corollary}[section]

\newtheorem{conj}{Conjecture}[section]

\theoremstyle{proof}

\numberwithin{equation}{section}

\begin{document} 
	\title[Infinitely many Counter examples of a conjecture of Franu\v{s}i\'c and Jadrijevi\'c]{Infinitely many Counter examples of a conjecture of Franu\v{s}i\'c and Jadrijevi\'c}
	\author{Shubham Gupta}
	\address{S. Gupta @Stat Math Unit, Indian Statistical Institute, 7 SJS Sansanwal Marg, New Delhi 110016, India.}
	\email{shubhamgupta2587@gmail.com}
	\date{\today}
	\maketitle
	
	\begin{abstract}
		Let $d$ be a square-free integer such that $d \equiv 15 \pmod{60}$ and the Pell's equation $x^2 - dy^2 = -6$ is solvable in rational integers $x$ and $y$. In this paper, we prove that there exist infinitely many Diophantine quadruples in $\mathbb{Z}[\sqrt{d}]$ with the property $D(n)$ for certain $n$'s. As an application of it,  we 
		`unconditionally' prove the existence of infinitely many rings $\mathbb{Z}[\sqrt{d}]$ for which the conjecture of Franu\v{s}i\'c and Jadrijevi\'c (Conjecture \ref{JadZ2017})  does `not' hold. This conjecture states a relationship between the existence of a Diophantine quadruple in $\mathcal{R}$  with the property $D(n)$ and the representability of $n$ as a difference of two squares in $\mathcal{R}$, where $\mathcal{R}$ is a commutative ring with unity.
		
	\end{abstract}

	\section{Introduction}
	
	Let $n \in \mathbb{Z}$. A set $\{a_1, a_2, \ldots, a_m\}$ of $m$ distinct positive integers is called a \textit{Diophantine $m$-tuple with the property $D(n)$} if it satisfies the following condition: $a_ia_j + n = x_{ij}^2$ for all $1 \leq i < j \leq m$, where $x_{ij} \in \mathbb{Z}$. Diophantine $m$-tuples have a long and rich history dating from the time of  Diophantus of Alexandria. Diophantus was the first who discovered these kinds of sets in rational numbers. For example, the set $\{1/16, 33/16, 17/4, 105/16\}$ has the property $D(1)$. The first Diophantine $4$-tuple with the property $D(1)$, namely $\{1, 3, 8, 120\}$, was discovered by Fermat.
	In 1969, by using the Baker's theory on linear forms in logarithms of algebraic numbers, Baker and Davenport \cite{BD1969} demonstrated that this quadruple cannot be extended to a Diophantine quintuple with the same property.  A `folklore' conjecture states that ``there does not exist any Diophantine $5$-tuple with the property $D(1)$''. In order to prove it, Dujella \cite{DU2004} in 2004 proved the non-existence of a Diophantine sextuple with the property $D(1)$ and also established the existence of at most finitely many Diophantine quintuples with the same property. Finally, in 2019, He, Togb\'e, and Ziegler \cite{HTZ2019} proved the folklore conjecture. 
	
	Now, we look at some established results for Diophantine $m$-tuples with the property $D(-1)$. In 1993, Dujella \cite{DU1993} gave the following conjecture: There is no Diophantine $4$-tuple with the property $D(-1)$. In this direction, Dujella and Fuchs \cite{DF2005} in 2005 demonstrated the non-existence of a  Diophantine $5$-tuple with the property $D(-1)$. In the same paper, they also showed that for any Diophantine quadruple $\{a, b, c, d\}$ with the property $D(-1)$ satisfying $a < b < c < d$,  then $a = 1$. Finally, in 2022, Bonciocat, Cipu, and Mignotte \cite{BCM2020} proved that there is no Diophantine quadruple with the property $D(-1)$. For more information on Diophantine $m$-tuples with the property $D(n)$, see \cite{BR1985, DU1993, BTF2019, DU21, CGH22} and references therein. Readers may also visit the webpage \url{https://web.math.pmf.unizg.hr/~duje/dtuples.html} created by Dujella for a brief survey on Diophantine $m$-tuples. 
	
	Now, let us generalize the definition of a Diophantine $m$-tuple from positive integers to any commutative ring $R$ with unity. Suppose that $n \in  R$. A set $\{a_1, a_2, \ldots, a_m\} \subset R \setminus \{0\}$ is called a \textit{ Diophantine $m$-tuple in $R$ with the property $D(n)$} if $a_ia_j + n = x_{ij}^2$ for all $i \neq j$ and for some $x_{ij} \in R$. In 1993, Dujella \cite{DU1993} demonstrated that there exists a Diophantine quadruple with the property $D(n)$ if and only if $n$ can be represented as a difference of two squares in rational integers, up to finitely many exceptions of $n$. Later on, Dujella \cite{DU1997} showed that this fact is also true for the ring of  Gaussian integers. Furthermore, this fact  holds for some rings of integers, namely: the ring of integers of  $\mathbb{Q}(\sqrt{d})$ for certain $d$'s (see \cite{MR2004, FR2004, FR2008, FR2009, FS2014, SO2013});  the ring of integers of $\mathbb{Q}(\sqrt[3]{2})$ (see \cite{FR2013} and \cite{MA2012}). Motivated by these examples, in 2019, Franu\v{s}i\'c and Jadrijevi\'c  proposed the following conjecture:
	\begin{conj}\cite[Conjecture 1]{FJ2019}\label{JadZ2017}
		Let $\mathcal{R}$ be a commutative ring of unity, and $n \in \mathcal{R} \setminus \{0\}$. A Diophantine quadruple in $\mathcal{R}$ with the property $D(n)$ exists if and only if $n$ is of the form $\alpha^2 - \beta^2$ for some $\alpha, \beta \in \mathcal{R}$, up to finitely many exceptions of $n$. 
	\end{conj}
	Moreover, in the same paper, they proved that this conjecture is valid for the ring of integers of the  bi-quadratic field $\mathbb{Q}(\sqrt{2}, \sqrt{3})$.
	
	In 2023, the first counter-example of Conjecture \ref{JadZ2017} has been discovered by the author, Chakraborty, and Hoque \cite{CGH2022}. They proved  that by assuming the truth of  Buniakovsky's conjecture \cite{BU1857}, there exist infinitely many quadratic rings of integers for which Conjecture \ref{JadZ2017} does `not' hold. Now, in the present paper, we unconditionally give a family of infinitely many quadratic rings of integers for which  Conjecture  \ref{JadZ2017} does not hold. 
	\subsection*{Main Results} We prove the following theorem regarding the existence of Diophantine quadruples in $\mathbb{Z}[\sqrt{d}]$ with the property $D(n)$ for certain $n$'s, where $d \equiv 15 \pmod{60}$ such that the Pell's equation 
	\begin{equation}\label{eq_norm_-6}
		x^2 - dy^2 = -6
	\end{equation} 
	is solvable in rational integers $x$ and $y$.

	\begin{thm}\label{main_result_2}
		Let $d$ be defined as the above, and let $m, k \in \mathbb{Z}$. There exist  infinitely many Diophantine quadruples in $\mathbb{Z}[\sqrt{d}]$ with the property $D(4m + 2 + 4k\sqrt{d})$ for even $m + k$.
	\end{thm}

	As a consequence of the aforementioned theorem, we show  in the last section that there exist infinitely many  $n$'s of the form $4m + 2 +  4k\sqrt{d}$, which cannot be represented as a difference of two squares in $\mathbb{Z}[\sqrt{d}]$. However, for these types of $n$'s, there exists a Diophantine quadruple in $\mathbb{Z}[\sqrt{d}]$ with the property $D(n)$ by using the above theorem.
	In this manner, we conclude that there exist infinitely many rings $\mathbb{Z}[\sqrt{d}]$ for which Conjecture \ref{JadZ2017} does not hold.    
	
	\section{Preliminaries} \label{pre}
	
	Let $ k, m \in \mathbb{Z}$, and let $d$ be a square-free rational integer. Now, we construct two disjoint subsets $S$ and $T$ of $\mathbb{Z}[\sqrt{d}]$ which are as follows:
	\begin{align*}
		S &= \{(2m+ 1 + 2k\sqrt{d}),~(4m + 4k\sqrt{d}),~(4m + (4k + 2)\sqrt{d}), ~(4m + 2 + 4k\sqrt{d})\}, \\ 
		T &= \mathbb{Z}[\sqrt{d}] \setminus S.
	\end{align*}
	
	\begin{lem} \cite[Propositions 5 and 6]{FR2008}
		Let $T$ and $d$ be defined as the above, with $d  \equiv 3 \pmod4$. If $n \in T$, then there does not exist any Diophantine quadruple in $\mathbb{Z}[\sqrt{d}]$ with the property $D(n)$.
	\end{lem}
	
	Now, we look at the existence of Diophantine quadruples in $\mathbb{Z}[\sqrt{d}]$ with the property $D(n)$ for some $n \in S$. To achieve this goal, we use the next two lemmas which can be derived from the definition of a Diophantine $m$-tuple.
	
	\begin{lem}{\cite[Lemma 2.5]{CGH2022}}\label{lem2.1}
		For any $n\in\mathbb{Z}[\sqrt{d}]$, a set $\mathcal{A} = \{a, b, a + b + 2r, a + 4b + 4r\}$ of non-zero and distinct elements forms a  Diophantine quadruple  in $\mathbb{Z}[\sqrt{d}]$ with the property $D(n)$, if
		$ ab + n = r^2$ and $3n = \alpha_1\alpha_2$ with $\alpha_1=a + 2r + \alpha$ and $ \alpha_2= a + 2r - \alpha$
		for some $ a, b, r, \alpha\in \mathbb{Z}[\sqrt{d}]$.
	\end{lem}

	\begin{lem}{\cite[Lemma 2.3]{CGH2022}}\label{lemp}
		For a square-free $d\in \mathbb{Z}$, let a set  $\{a_1, a_2, a_3, a_4\}$ be a Diophantine quadruple in  $\mathbb{Z}[\sqrt{d}]$ with the property $D(n)$. Then for any non-zero $w\in \mathbb{Z}[\sqrt{d}]$,  the set $\{wa_1, wa_2, wa_3, wa_4\}$ forms a  Diophantine quadruple in $\mathbb{Z}[\sqrt{d}]$ with the property $D(w^2n)$.
	\end{lem}

	For examining the condition of non-zero and distinctness of the set $\mathcal{A}$ (defined in Lemma \ref{lem2.1}), we use the following result:
	
	\begin{lem}{\cite[Lemma 2.4]{CGH2022}}\label{lemm1} Suppose that $a_1, a_2, b_1, b_2, c_1, c_2, d_1, d_2, e_1 \in \mathbb{Z}$ satisfying $a_1a_2b_1 \neq 0$. Then
		the following system of simultaneous equations
		\begin{align}\label{eQ7}
			\begin{cases}
				a_1x^2 + b_1y^2 + c_1x + d_1y + e_1 = 0\\
				a_2xy + b_2x + c_2y + d_2 = 0
			\end{cases}
		\end{align}
		has at most finitely many solutions $x$ and $y$ in $\mathbb{Z}$.
	\end{lem} 
	
	%

	\subsection*{Notations} Let us fix some notations which will be used throughout this paper.
	\begin{itemize}
		\item [] (i)   Let $d$ be a square-free rational integer such that $d \equiv 15 \pmod{60}$ and \eqref{eq_norm_-6} is solvable in rational integers  $x$ and $y$.
		\item [] (ii)  $(a, b) := a + b\sqrt{d} \in \mathbb{Z}[\sqrt{d}]$.
		And  for $k \in \mathbb{Z}$, $k(a, b) = (ka, kb)$. 
		\item  [] (iii) 	Let $\alpha=(a, b)$. We define the norm of an element $\alpha$  by 
		$$\text{Nm}(\alpha) := (a, b)(a, -b).$$ 
		\item[] (iv) $a~|~b$ denotes that $a$ divides $b$.
	\end{itemize}
	
	\section{Certain Elements of $\mathbb{Z}[\sqrt{d}]$}
	
	In this section, we will see how the elements  of norms $1$ and $-6$ in  $\mathbb{Z}[\sqrt{d}]$ look like. 
	
	\begin{lem} \label{norm_1}
		Let $d$ be as defined in Notations of $\S$ \ref{pre}. Then the Pell's equations $x^2 - dy^2 = \pm 2$ are not  solvable in rational integers $x$ and $y$.
	\end{lem}
	
	\begin{proof}
		We have $d \equiv 15 \pmod{60}$. Suppose that  the equations 
		$
		x^2 - dy^2 = \pm 2
		$
		are solvable in $\mathbb{Z}$.  
		Then 
		$$
		x^2 \equiv \pm 2 \pmod5,
		$$
		which is not possible. This completes the proof.
	\end{proof}
	
	The next lemma describes the elements of norm $-6$.
	
	\begin{lem} \label{norm_6}
		Let \eqref{eq_norm_-6} be solvable in rational integers. The following statements are true:\\
		$(i)$ All solutions $(x, y)$ of this equation are of the form   $(6\alpha \pm 3, 6\beta \pm 1)$ for some $\alpha, \beta \in \mathbb{Z}$ and there are infinitely many such $(x, y)$. 	\\
		$(ii)$ Moreover, among these $\alpha$ and $\beta$, there exist infinitely many $\alpha$ and $\beta$ such that $\alpha + \beta$ is an even and an odd integer.
	\end{lem}
	
	\begin{proof}
		(i) Consider \eqref{eq_norm_-6} and reduce it  modulo 2. This gives that
		\begin{equation}\label{eq_mod_2}
			x, y \equiv 1 \pmod2.
		\end{equation}
		Now, reduce  \eqref{eq_norm_-6}  at modulo 3. We obtain $3|x$. Since $3 \nmid y$, otherwise $9|6$. Thus, 
		\begin{equation}\label{eq_mod_3}
			x \equiv 0 \pmod3 \text{~~and~~} y \equiv \pm 1 \pmod3. 
		\end{equation} 
		Combining \eqref{eq_mod_2} and \eqref{eq_mod_3}, this implies that  $x = 6\alpha \pm 3$ and $y = 6\beta \pm 1$ for some $\alpha, \beta \in \mathbb{Z}$. 
		
		Consider $\gamma\delta$, where $\text{Nm}(\gamma) = -6$, $\text{Nm}(\delta) = 1$. Due to the existence of infinitely many $\delta$'s, there are infinitely many elements of norm $-6$.

		\noindent
		(ii) Let $(x,y)$ be a solution of \eqref{eq_norm_-6}. Then $(\pm x, \pm y)$ will also be  solutions of \eqref{eq_norm_-6}.  Suppose that $\alpha$ is an even (odd) number such that $x = 6\alpha + 3$. Then for the case of  $6\alpha' + 3 = -x$, $\alpha'$ will be odd (even). So for any $\beta$, we can choose $\alpha$ such that  $\alpha + \beta$ is an even and odd integer. Since there exist infinitely many solutions of \eqref{eq_norm_-6},  there are infinitely many $\alpha$ and $\beta$ such that $\alpha + \beta$ can be an even and odd integer.
	\end{proof}
	Utilizing Lemma \ref{norm_6}, we can deduce the following corollary:
	\begin{cor}\label{d}
		Let $d$ be as defined in Notations of $\S$ \ref{pre}. Then
		$d \equiv 15 \pmod{360}$.
	\end{cor}
	
	\begin{proof}
		From Lemma \ref{norm_6}, if \eqref{eq_norm_-6} is solvable in $\mathbb{Z}$, then $(x, y)$ will be of the form $(6\alpha \pm 3, 6\beta \pm 1)$ for some $\alpha, \beta \in \mathbb{Z}$, i.e.
		$$
		(6\alpha \pm 3)^2 - d(6\beta \pm 1)^2 = -6.
		$$
		Reduce this equation at modulo 72. Using $3|d$, $d \equiv 15 \pmod{72}$. Since $d \equiv 15 \pmod{60}$ and $lcm(60, 72) = 360$, we get that $d \equiv 15 \pmod{360}$, which is the desired result.
	\end{proof}
	
	The following lemma describes the form of  elements of norm $1$ in  $\mathbb{Z}[\sqrt{d}]$.

	\begin{lem} \label{lem_norm_1} 
		Suppose that $d$ is as defined in Notations of $\S$ \ref{pre}. The following statements are true:\\
		$(i)$ There exists a solution  $(\gamma', \delta')$ of  the  equation 
		\begin{equation}\label{eq_1}
			x^2 - dy^2 = 1 
		\end{equation}
		such that  $\delta'$ is an odd integer and $\gamma'$ has the form $6\alpha \pm 4$, where $\alpha \in \mathbb{Z}$. \\
		$(ii)$  A fundamental solution  of \eqref{eq_1} has the form either   $(6\alpha \pm 4, 6\beta \pm 1)$ or $(6\alpha \pm 4, 6\beta + 3)$, where $\alpha, \beta \in \mathbb{Z}$.
	\end{lem}
	
	\begin{proof}
		(i) Let $(\gamma, \delta)$ be a solution of \eqref{eq_norm_-6}. Take $\gamma' = (\gamma^2 + 3)/3$ and $\delta' = (\gamma \delta)/3$.  From Lemma \ref{norm_6}(i), $3|\gamma$. This gives that $\gamma', \delta' \in \mathbb{Z}$. Consider the expression
		\begin{align*}
			\gamma'^2 - d\delta'^2 = \dfrac{1}{9}[(\gamma^2 + 3)^2 - d(\gamma \delta)^2] &= \dfrac{1}{9}[\gamma^2(\gamma^2 - d\delta^2) + 6\gamma^2 + 9] \\
			&= \dfrac{1}{9}[-6(\gamma^2) + 6\gamma^2 + 9]\text{\hspace{0.2cm} (by using \eqref{eq_norm_-6})}\\
			&= 1.
		\end{align*}
		Thus, $(\gamma', \delta')$ is an integral  solution of \eqref{eq_1}. From Lemma \ref{norm_6}(i), $\gamma$ is an odd integer and $3|\gamma$, so $\gamma'$ will be of the form $6\alpha \pm 4$, where $\alpha \in \mathbb{Z}.$ Indeed, 
		if $(x, y)$ is a solution of \eqref{eq_1}, then $(\pm x, \pm y)$ are also its solutions.
		
		Since both $\gamma$ and $\delta$ are odd by using Lemma \ref{norm_6}, so $\delta'$ is an odd integer.
		
		\noindent
		(ii) It is well known that there exists a solution  $s = (\gamma_1', \delta_1')$ of \eqref{eq_1} such that all the solutions of  \eqref{eq_1} have the form $\pm s^n$ (see \cite[Theorem 3.20]{MOL2011}), where $n \in \mathbb{Z}$. We call that $s$ is a \textit{fundamental solution} of \eqref{eq_1}. Suppose that $\gamma_1'$ is odd and $\delta_1'$ is even. Then for every $n$, we get that $\pm s^n = (t_1, t_2)$ with odd $t_1$ and even $t_2$. But, from the first part of Lemma  \ref{lem_norm_1}, we get a contradiction. Therefore, 
		\begin{align}\label{gam}
			\gamma_1' \equiv 0 \pmod2,~~~~~~ \delta_1' \equiv 1 \pmod2.
		\end{align}		
		
		Consider \eqref{eq_1} at modulo 3. This gives that 
		$
		\gamma_1' \equiv \pm 1 \pmod{3}$ and $\delta_1' \equiv 0, \pm 1 \pmod{3}.
		$
		Hence,
		by using \eqref{gam}, $s$ has the form of either  $(6\alpha \pm 4, 6\beta \pm 1)$ or $(6\alpha \pm 4, 6\beta + 3)$, due to $6\beta - 3 = 6(\beta - 1) + 3$. This completes the proof.
	\end{proof}
	The following two examples support the fact that these two types of fundamental solutions of \eqref{eq_1} given in Lemma  \ref{lem_norm_1}(ii) exist.
	
	\subsection*{Example 1: }Let $d = 15$. Then,  fundamental solutions for \eqref{eq_1} and \eqref{eq_norm_-6} are $(4, 1)$ and $(3, 1)$, respectively. Here, a fundamental solution of \eqref{eq_1} is of the form $(6\alpha + 4, 6\beta + 1)$.
	
	\subsection*{Example 2: }Let $d = 735$. Then,  fundamental solutions for \eqref{eq_1} and \eqref{eq_norm_-6} are $(244, 9)$ and $(27, 1)$, respectively. Here, a fundamental solution of \eqref{eq_1} is of the form $(6\alpha + 4, 6\beta + 3)$.

	\section{Proof of Theorem \ref{main_result_2}}
	Let $n = 2(2m +1 , 2k)$. By using Lemma \ref{norm_6}(i), $-6$ can be written as 
	\begin{align}\label{fac-6}
		-6 = (6\alpha + 3, -6\beta - 1)(6\alpha + 3, 6\beta + 1)  \text{~~~or~~~} (6\alpha + 3, -6\beta + 1)(6\alpha + 3, 6\beta - 1)
	\end{align}
	for some $\alpha, \beta \in \mathbb{Z}$. Consider the first type of factorization of $-6$ which is given in the above equation. This gives that  $3n$ can be factorized in the following way:
	\begin{align*}
		3n &= (-1)(-6)(2m + 1, 2k)\\
		&=  (-6\alpha - 3, 6\beta + 1)(6\alpha + 3, 6\beta + 1)(2m + 1, 2k)
		\\
		&= \Big(-6\alpha - 3, 6\beta + 1\Big)\Big(12\alpha m + 6\alpha + 6m + 3 + d(12\beta k + 2k),~
	 12 \alpha k + 6k + 12\beta m+ 6\beta + 2m +1\Big).
	\end{align*}	
	Suppose that the first factor on the right hand side of the above expression is $\alpha_1$ and the second one is $\alpha_2$. Then, from Lemma \ref{lem2.1}
	$$
	\alpha_1 + \alpha_2 = 2a + 4r = \Big(12\alpha m + 6m + d(12\beta k + 2k), 12 \alpha k + 6k + 12\beta m+ 12\beta + 2m +2\Big). 
	$$
	By dividing by 2,
	$$
	a + 2r = \Big(6\alpha m + 3m + d(6\beta k + k), 6 \alpha k + 3k + 6\beta m+ 6\beta + m +1\Big). 
	$$
	From Lemma \ref{lem_norm_1}(ii), we choose $a = (2a_1, 2b_1 + 1)$ with its norm being 1, where $a_1, b_1 \in \mathbb{Z}$. Then the above equation gives that
	$$
	r = \Big(3\alpha m + 3d\beta k) - a_1 + (3m + dk)/2, 3 \alpha k + 3\beta m + (m + 3k)/2\Big).
	$$
	By hypothesis, $m + k$ is an even integer. So, $r \in \mathbb{Z}[\sqrt{d}]$. Now, we calculate $b$ using the formula 
	$
	b = \dfrac{r^2 - n}{a},
	$
	from Lemma \ref{lem2.1}.
	Since  $a$ is an unit in $\mathbb{Z}[\sqrt{d}]$, $b \in \mathbb{Z}[\sqrt{d}]$. Till now, we have $a, b, r \in \mathbb{Z}[\sqrt{d}]$. Using Lemma  \ref{lemm1}, only for at most finitely many $a$'s, the set $\mathcal{A}$ (given in Lemma \ref{lem2.1}) contains a zero or two equal elements. But we have infinitely many choices of $a$.
	Hence by utilizing Lemma \ref{lem2.1}, there exist infinitely many Diophantine quadruples $\mathcal{A}$ in $\mathbb{Z}[\sqrt{d}]$ with the property $D(n)$.
	
	Analogously, we may choose the second type of factorization of $-6$, which is given in \eqref{fac-6}, to get  Diophantine quadruples in $\mathbb{Z}[\sqrt{d}]$ with the property $D(n)$. 
	
	This completes the proof. \qed
	
	\section{Counter-examples of Conjecture \ref{JadZ2017}}\label{ce}
	Let $d$ be as defined in Notations of \textsection \ref{pre}. Throughout this section, we assume that $n = (4m + 2, 4k) = 2(2m + 1, 2k)$ satisfying
	\begin{equation}\label{eq_(2m)}
		(2m + 1)^2 - d(2k)^2 = 1,
	\end{equation}
	where $m, k \in \mathbb{Z}$. Let $(\gamma, \delta)$ be a fundamental solution of \eqref{eq_1}, and let $t \in \mathbb{Z}$. 
	Using Lemma \ref{lem_norm_1}(ii), $\gamma$ is an even integer and $\delta$ is an odd integer. 
	Now taking $(\gamma, \delta)^{2t} = (\gamma'', \delta'')$, we get that for every $t$, $\gamma''$ is odd and $\delta''$ is even.
	Then corresponding to these $\gamma''$  and $ \delta''$, there exist $m$ and $k$, respectively,  which satisfy \eqref{eq_(2m)}.
	Therefore, $n$ $ = 2(\gamma, \delta)^{2t}$. 
%
%
	
	\begin{lem}\label{diff}
		Let $n$ be as defined at the beginning of \textsection \ref{ce}, i.e., $n = 2(2m + 1, 2k)$ satisfying \eqref{eq_(2m)}, where $m, k \in \mathbb{Z}.$ Suppose that $d$ is as given in Notations of \textsection \ref{pre}.
		Then  $n$ cannot be written as a difference of two squares in $\mathbb{Z}[\sqrt{d}]$.
	\end{lem}

	\begin{proof}
		We follow \cite[Proposition 5]{DF2007} to prove the lemma. Suppose that $n$ is expressible as a difference of two squares in $\mathbb{Z}[\sqrt{d}]$. Then  for some rational integers $x_1, x_2, y_1$, and $y_2$, we have:
		\begin{equation}\label{sq1}
			n = (4m + 2, 4k) = (x_1, y_1)^2 - (x_2, y_2)^2 = (x_1^2 + dy_1^2 - x_2^2 - dy_2^2, 2x_1y_1 - 2x_2y_2).
		\end{equation} 
		This gives that
		\begin{align}
			x_1^2 + dy_1^2 - x_2^2 - dy_2^2 &= 4m+ 2 \label{4k+2}\\
			x_1y_1 - x_2y_2 &= 2k \label{2k}.  
		\end{align}
		
		First, we show that $y_1$ and $y_2$ have different parity. Suppose not, then $y_1$ and $y_2$ are both even or odd. Therefore, $y_1^2 - y_2^2 \equiv 0 \pmod{4}$. Using it along with \eqref{4k+2}, we get that $x_1^2 - x_2^2 \equiv 2 \pmod{4}$, and this is not possible for any  $x_1$ and $x_2$. Thus,
		\begin{equation}\label{y1+y2}
			y_1 + y_2 \equiv 1 \pmod{2}.
		\end{equation}
		Next, by using the above equation and $d \equiv 3 \pmod{4}$, we get that either $x_1^2 - x_2^2 \equiv 1$ or $3 \pmod{4}$,  from \eqref{4k+2}. So, we find that
		$x_1$ and $x_2$ also have different parity, i.e.,
		\begin{equation}\label{sq2}
			x_1 + x_2 \equiv 1 \pmod{2}.
		\end{equation}	
		
		Let $x_1 = x_2 + \alpha$ and $y_1 = y_2 + \beta$ for some rational integers $\alpha$ and $\beta$. 
		Then by using \eqref{4k+2} and \eqref{2k}, we get the following equations:
		\begin{align}\label{leq}
			\alpha x_2 + d\beta y_2 = 2m +1 - \dfrac{\alpha^2 + d\beta^2}{2} \hspace{0.5cm}\text{and} \hspace{0.5cm}
			\beta x_2 + \alpha y_2 = 2k - \alpha \beta,
		\end{align}
		respectively.
		Now, we use \eqref{y1+y2} and \eqref{sq2} to get that
		\begin{equation}\label{alb}
			\alpha, \beta \equiv 1 \pmod{2}.
		\end{equation}
		The above gives that 
		\begin{equation}\label{in_integer}
			\dfrac{\alpha^2 + d\beta^2}{2} \in \mathbb{Z},
		\end{equation}  
		since $d \equiv 3 \pmod{4}$.	Moreover, using \eqref{alb} and $d \equiv 3 \pmod{4}$, we get that $\alpha^2 - d\beta^2 \not\equiv 0 \pmod{4}$.
		So, 
		\begin{equation}\label{neq_zero}
			\alpha^2 - d\beta^2 \neq 0.
		\end{equation} 
		
		From the above equation, the determinant of the system of linear equations \eqref{leq} in $x_2$ and $y_2$, which is $\alpha^2 - d\beta^2$, is not equal to zero. So, \eqref{leq}  has the unique solution, which is given by:
		\begin{align}
			x_2 ~&=~ \dfrac{\left(\left(2m +1 - \dfrac{\alpha^2 + d\beta^2}{2}\right)\alpha - \left(2k - \alpha \beta\right)d\beta\right)}{(\alpha^2 - d\beta^2)},\label{x_22}
			\\
			y_2 ~&=~ \dfrac{\left(\left(2k - \alpha \beta\right)\alpha - \left(2m +1 - \dfrac{\alpha^2 + d\beta^2}{2}\right) \beta\right)}{(\alpha^2 - d\beta^2)}. \label{y_22}
		\end{align} 	
		Note that $x_2 \in \mathbb{Z}$ and 	$\dfrac{\alpha^2 + d\beta^2}{2} \in \mathbb{Z}$ (from \eqref{in_integer}), \eqref{x_22} gives  that
		\begin{align}
			(\alpha^2 - d\beta^2) ~&\Bigg|~ \Bigg(\left(2m +1 - \dfrac{\alpha^2 + d\beta^2}{2}\right)\alpha - \left(2k - \alpha \beta\right)d\beta\Bigg) \nonumber \\
			~\Rightarrow~	(\alpha^2 - d\beta^2) ~&\Bigg|~ 2 \times \Bigg(\left(2m +1 - \dfrac{\alpha^2 + d\beta^2}{2}\right)\alpha - \left(2k - \alpha \beta\right)d\beta\Bigg)\nonumber\\
			~\Rightarrow~
			(\alpha^2 - d\beta^2) ~&\Bigg|~       \Bigg(2 \times \Big((2m + 1) \alpha - d(2k)\beta \Big) - \alpha \Big(\alpha^2 - d\beta^2 \Big)\Bigg) \hspace{0.5cm} \text{(by rearranging the terms)} \nonumber\\
			~\Rightarrow~	(\alpha^2 - d\beta^2) ~&\Big|~ 2 \times \Big((2m + 1) \alpha - d(2k)\beta \Big) \label{er1}. 
		\end{align} 
		Similarly, due to $y_2 \in \mathbb{Z}$ and $\dfrac{\alpha^2 + d\beta^2}{2} \in \mathbb{Z}$ (from \eqref{in_integer}), \eqref{y_22} implies that
		\begin{align}
			(\alpha^2 - d\beta^2) ~&\Bigg|~  \Bigg(\left(2k - \alpha \beta\right)\alpha - \left(2m +1 - \dfrac{\alpha^2 + d\beta^2}{2}\right) \beta\Bigg)
			\nonumber \\
			~\Rightarrow~ 	(\alpha^2 - d\beta^2) ~&\Bigg|~ 2 \times \Bigg(\left(2k - \alpha \beta\right)\alpha - \left(2m +1 - \dfrac{\alpha^2 + d\beta^2}{2}\right) \beta\Bigg) \nonumber \\
			~\Rightarrow~ 	(\alpha^2 - d\beta^2) ~&\Bigg|~ \Bigg(2 \times \Big((2k)\alpha - (2m + 1)\beta \Big) - \beta 
			\Big(\alpha^2 - d\beta^2\Big)\Bigg) \hspace{0.5cm} \text{(by rearranging the terms)} \nonumber \\
			~\Rightarrow~ (\alpha^2 - d\beta^2) ~&\Big|~ 2 \times \Big((2k)\alpha - (2m + 1)\beta \Big). \label{er2}
		\end{align} 
		
		Now, we prove the following two claims:
		\\	
		\noindent
		\textbf{Claim 1:} $(\alpha^2 - d\beta^2) ~|~ 2 \alpha$.
		
		Equations	\eqref{er1} and \eqref{er2} imply that  
		\begin{align*}
			(\alpha^2 - d\beta^2) &~\Bigg|~ \Bigg(2m + 1\Bigg) \times   \Bigg(2 \times \Big((2m + 1) \alpha - d(2k)\beta \Big)\Bigg) \hspace{0.5cm} \text{~and~}\\			 
			(\alpha^2 - d\beta^2) &~\Bigg|~ \Bigg(2kd\Bigg)  \times \Bigg(2  \times \Big((2k)\alpha - (2m + 1)\beta \Big)\Bigg), 
		\end{align*}
		respectively.  Furthermore, the above implies  that 
		\begin{align*}
			(\alpha^2 - d\beta^2) ~&\Bigg|~  \Bigg(\Big(2m + 1 \Big)  \Big(2 \times \Big((2m + 1) \alpha - d(2k) \beta \Big)\Big) -  \Big(2kd \Big)\Big(2 \times \Big((2k)\alpha - (2m + 1)\beta \Big)\Big)\Bigg).
		\end{align*} 
		On solving the right-hand side of the above equation, we get that 	
		\begin{align}\label{mi1}
				(\alpha^2 - d\beta^2) ~\Big|~ 2 \times \Big((2m + 1)^2 - d (2k)^2\Big) \times \alpha.
		\end{align}
		Now $ (2m + 1)^2 - d(2k)^2 = 1$ because of our hypothesis $n = 2 (2m + 1, 2k)$ satisfying \eqref{eq_(2m)}. So, \eqref{mi1} gives that
		\begin{align}\label{a1}
			(\alpha^2 - d\beta^2) ~&|~ 2\times 1 \times \alpha ~~\Rightarrow~~ 	(\alpha^2 - d\beta^2) ~|~ 2 \alpha.
		\end{align}
		This proves the claim.
		
		\noindent
		\textbf{Claim 2: } $(\alpha^2 - d\beta^2) ~|~ 2\beta $.
		
		As we have proved Claim 1, we now proceed to prove this claim.	Equations \eqref{er1} and \eqref{er2} imply
		\begin{align*}
			(\alpha^2 - d\beta^2) ~&\Bigg|~ \Bigg(2k\Bigg) \times \Bigg(2 \times \Big((2m + 1) \alpha - d(2k)\beta \Big)\Bigg), \hspace{0.5cm} \text{and} \\
			(\alpha^2 - d\beta^2) ~&\Bigg|~ \Bigg(2m + 1\Bigg) \times \Bigg(2 \times \Big((2k)\alpha - (2m + 1)\beta \Big)\Bigg), 
		\end{align*}
		respectively. Furthermore, the above implies that
		\begin{align}\label{mi3}
			(\alpha^2 - d\beta^2) ~&\Bigg|~ \Bigg(\Big(2k\Big)\Big(2 \times \Big((2m + 1) \alpha - d(2k)\beta \Big)\Big) - \Big(2m + 1\Big)\Big(2 \times \Big((2k)\alpha - (2m + 1)\beta \Big)\Big)\Bigg).
		\end{align}
		On solving the right-hand side of \eqref{mi3}, we get that
		\begin{align}\label{mi2}
			(\alpha^2 - d\beta^2) ~\Big|~ 2 \times \Big((2m + 1)^2 - d (2k)^2\Big) \times \beta.
		\end{align}
		Now  $ (2m + 1)^2 - d(2k)^2 = 1$ because of our  hypothesis $n = 2 (2m + 1, 2k)$ satisfying \eqref{eq_(2m)}. So, \eqref{mi2} gives that
		\begin{align}
			(\alpha^2 - d\beta^2) ~|~ 2 \times 1 \times \beta ~\Rightarrow~ (\alpha^2 - d\beta^2) ~|~ 2 \beta  \label{a2},
		\end{align}
		which proves the claim.
		
		Now, let $q = \gcd(\alpha, \beta)$. Then there exist $x, y \in \mathbb{Z}$ such that $q = \alpha x + \beta y$. Using \eqref{a1} and \eqref{a2}, this gives that  
		\begin{equation}\label{sqq1}
			(\alpha^2 - d\beta^2) ~|~ 2q.
		\end{equation}
		Since $q^2 ~|~ (\alpha^2 - d\beta^2)$, \eqref{sqq1} implies that 
		\begin{equation}\label{q211}
			q^2 ~|~ 2q. 
		\end{equation}		  
		From \eqref{alb}, $q$ is an odd integer. Then, \eqref{q211} gives $q = 1$.  Thus, $\alpha^2 - d\beta^2 = \pm 1, \pm 2$, from \eqref{sqq1}.
		By using \eqref{alb}, $\alpha^2 - d\beta^2 = \pm 1$ are not possible. Therefore, $\alpha^2 - d\beta^2 = \pm 2$. 
		However, Lemma \ref{norm_1} says that there do not exist any $\alpha$ and $\beta$ which satisfy the last equation. Hence, our supposition is wrong and this completes the proof.		
	\end{proof}
%
	Taking $m = k = 0$ in Theorem \ref{main_result_2}, we get that there exists a Diophantine quadruple in $\mathbb{Z}[\sqrt{d}]$ with the property $D(2)$. After that, we use Lemma \ref{lemp} to achieve that there exists a Diophantine quadruple in $\mathbb{Z}[\sqrt{d}]$	with the property $D(2(\gamma, \delta)^{2t}).$ On the other hand, from Lemma \ref{diff}, $n$  cannot be written as a difference of two squares in $\mathbb{Z}[\sqrt{d}]$. To conclude, Conjecture \ref{JadZ2017} does not hold for such rings $\mathbb{Z}[\sqrt{d}]$.

	Now, we prove that there are infinitely many such rings   $\mathbb{Z}[\sqrt{d}]$.  
	Take $x = 60\alpha + 3$, $y =1$, and 
	\begin{equation}\label{eq_d}
		d = 360(10\alpha^2 + \alpha) + 15,
	\end{equation}
	where  $\alpha \in \mathbb{Z}$. Note that  for these values of $x, y,$ and  $d$, we get that 
	$d \equiv 15 \pmod{360}$ and \eqref{eq_norm_-6} is solvable in rational integers for every $\alpha \in \mathbb{Z}$. This gives that all the conditions which are imposed on $d$  are satisfied. Therefore, $d$ can be taken as in \eqref{eq_d}. Hence, there exist  infinitely many rings $\mathbb{Z}[\sqrt{d}]$. 
	\subsection*{Example 3:} Let $d = 15$, and let  $n = 2(4 + \sqrt{15})^{2t}$, where $t \in \mathbb{Z}$. Utilizing Theorem \ref{main_result_2}, for $t = 0$, the set 
	$$
	\mathcal{B} = \{(4 + \sqrt{15}), (8 + 2\sqrt{15}), (8 - \sqrt{15}), (28  -7\sqrt{15})\}
	$$
	is  a Diophantine quadruple in $\mathbb{Z}[\sqrt{15}]$ with the property $D(2).$ Now, we use Lemma \ref{lemp} to get that for every $t$, the set 
	\begin{align*}
		(4 + \sqrt{15})^{t}\mathcal{B} = \{(4 + \sqrt{15})^{t}(4 + \sqrt{15}), (4 + \sqrt{15})^{t}(8 + 2\sqrt{15}), 
		(4 + \sqrt{15})^{t}(8 - \sqrt{15}), \\
		(4 + \sqrt{15})^{t}(28  -7\sqrt{15})\}
	\end{align*}
	is  a Diophantine quadruple in $\mathbb{Z}[\sqrt{15}]$ with the property $D(2(4 + \sqrt{15})^{2t})$. However, from Lemma \ref{diff}, $n$ cannot be written as a difference of two squares in $\mathbb{Z}[\sqrt{15}]$.

	\section*{Acknowledgments}
	
	The author thanks  Prof. Kalyan Chakraborty, Prof. Shanta Laishram, and Dr. Azizul Hoque for their support and fruitful discussions. Additionally, the author appreciates the hospitality provided by Indian Statistical Institute Delhi (ISI Delhi), where the work has been done. 
	The author expresses gratitude to the anonymous referee for their valuable suggestions and remarks, which have significantly enhanced the exposition of this paper.


\begin{thebibliography}{99}
		
		\bibitem{MR2004} F. S. Abu Muriefah and A. Al Rashed, \textit{Some Diophantine quadruples in the ring $\mathbb{Z}[\sqrt{-2}]$}, Math. Commun. \textbf{9} (2004), 1-8.
		
		\bibitem{BR1985} E. Brown, {\it Sets in which $xy + k$ is always a square}, Math. Comp. {\bf 45} (1985), 613-620.
		
		\bibitem{BU1857} V. Buniakovsky, {\it Sur les diviseurs numeriques invariables des fonctions rationnelles entieres}, Mem Acad. Sci. St Petersburg {\bf 6} (1857), 305-329.
		
		\bibitem{BCM2020} N. C. Bonciocat, M. Cipu, and M. Mignotte, {\it There is no Diophantine $D(-1)$-quadruple}, J. London Math. Soc. {\bf 105} (2022), 63-99.
		
		\bibitem{BD1969} A. Baker and H. Davenport, {\it The equations $3x^2 - 2 = y^2$ and $8x^2 - 7 = z^2$}, Quart. J. Math. Oxford Ser. (2) {\bf 20} (1969), 129-137.
		
		\bibitem{BTF2019} M. Bliznac Trebje\v{s}anin and A. Filipin, {\it Nonexistence of $D(4)$-quintuples}, J. Number Theory {\bf 194} (2019), 170-217.
		
		
		\bibitem{CGH2022} K. Chakraborty, S. Gupta, and A. Hoque, \textit{On a conjecture of Franu\v si\'c and Jadrijevi\' c: Counter-examples}, Results Math. {\bf 78} (2023), Article no. 18, 14pp.
		
		
		\bibitem{CGH22} K. Chakraborty, S. Gupta, and A. Hoque, {\it Diophantine triples with the property $D(n)$ for distinct $n$'s}, Mediterr. J. Math. {\bf 20} (2023), Article no. 31, 13pp.
		
		
		\bibitem{DU1993} A. Dujella, {\it Generalization of a problem of Diophantus}, Acta Arith. \textbf{65} (1993), 15-27.
		
		\bibitem{DU1997} A. Dujella, {\it The problem of Diophantus and Davenport for Gaussian integers}, Glas. Mat. Ser. III {\bf 32} (1997), 1-10.	
		
		
		\bibitem{DU2004} A. Dujella, {\it There are only finitely many Diophantine quintuples}, J. Reine Angew. Math. \textbf{566} (2004), 183-214.
		
		\bibitem{DU21}  A. Dujella, \textit{Number Theory}, \v{S}kolska knjiga, Zagreb, 2021.
		
		\bibitem{DF2005} A. Dujella and C. Fuchs, \textit{Complete solution of a problem of Diophantus and Euler,} J. London Math. Soc. \textbf{71} (2005), 33-52.
		
		\bibitem{DF2007} A. Dujella and Franu\v si\'c, \textit{On differences of two Squares in some quadratic fields},  Rocky Mountain J. Math. \textbf{37} (2007), 429-453. 
		
		
		\bibitem{FR2004} Z. Franu\v si\'c, \textit{Diophantine quadruples in the ring $\mathbb{Z}[\sqrt{2}]$}, Math. Commun. \textbf{9} (2004), 141-148. 
		
		
		
		\bibitem{FR2008} Z. Franu\v si\'c, \textit{Diophantine quadruples in $\mathbb{Z}[\sqrt{4k + 3}]$}, Ramanujan J. \textbf{17} (2008), 77-88.
		
		\bibitem{FR2009} Z. Franu\v si\'c, \textit{A Diophantine problem in $\mathbb{Z}[\sqrt{(1 + d)/2}]$}, Studia Sci. Math. Hungar. \textbf{46} (2009), 103-112.
		
		\bibitem{FR2013} Z. Franu\v si\'c, \textit{Diophantine quadruples in the ring of integers of the pure cubic field $\mathbb{Q}(\sqrt[3]{2})$}, Miskolc Math.
		Notes \textbf{14} (2013), 893-903.
		
		\bibitem{FJ2019} Z. Franu\v si\'c and B. Jadrijevi\' c, \textit{$D(n)$-quadruples in the ring of integers of $\mathbb{Q}(\sqrt{2}, \sqrt{3})$}, Math. Slovaca \textbf{69} (2019), 1263-1278.
		
		\bibitem{FS2014} Z. Franu\v si\'c and I. Soldo, \textit{The problem of Diophantus for integers of $\mathbb{Q}(\sqrt{-3})$}, Rad Hrvat. Akad. Znan. Umjet. Mat. Znan. \textbf{18} (2014), 15-25.
		
		\bibitem{HTZ2019} B. He, A. Togb\'e, and V. Ziegler, {\it There is no Diophantine quintuple}, Trans. Amer. Math. Soc. \textbf{371} (2019), 6665-6709.
		
		
		
		\bibitem{MA2012} Lj. Juki\'c Mati\'c, {\it Non-existence of certain Diophantine quadruples in rings of integers of pure cubic fields}, Proc. Japan Acad. Ser. A  Math. Sci. {\bf 88} (2012), no. 10, 163-167.
		
		\bibitem{MOL2011} R. A. Mollin, \textit{Algebraic number theory}, Second edn., Discrete
		Math. Appl. (Boca Raton), CRC  Press, Boca Raton, FL, 2011 Zbl0930.11001.
		
		\bibitem{SO2013} I.  Soldo, \textit{On the existence of Diophantine quadruples in $\mathbb{Z}[\sqrt{-2}]$}, Miskolc Math. Notes \textbf{14} (2013), 265-277.
		
		
		
	\end{thebibliography}
\end{document}